\documentclass[reqno]{amsart}
\usepackage{amssymb,amsmath,amsthm,amscd,latexsym,amsfonts}
\usepackage{xy}
\xyoption{all}
\usepackage{mathtools}
\usepackage[T1]{fontenc}
\usepackage{graphicx}
\usepackage{dsfont}
\usepackage{amsaddr}
\usepackage{cite}
\usepackage{bbm}
\usepackage{xcolor}
\usepackage{enumitem}

\usepackage[a4paper,top=3cm, bottom=3cm, left=3cm, right=3cm]{geometry}

\newtheorem{teo}{Theorem}

\newtheorem{rem}{Remark}

\begin{document}
    \title[Leibniz algebras constructed by Witt algebras]
    {Leibniz algebras constructed by Witt algebras}

    \author{L.M.~Camacho\textsuperscript{1}, B.A.~Omirov\textsuperscript{2}, T.K.~Kurbanbaev\textsuperscript{3}}
    \address{\textsuperscript{1} Dpto. Matem\'{a}tica Aplicada I.
        Universidad de Sevilla. Avda. Reina Mercedes, 41012 Sevilla, Spain, lcamacho@us.es}
    \address{\textsuperscript{2} National University of Uzbekistan, 100174, Tashkent, Uzbekistan, omirovb@mail.ru}
    \address{\textsuperscript{3} Institute of Mathematics of Uzbek Academy of Sciences, 100125, Tashkent, Uzbekistan, tuuelbay@mail.ru}

    \thanks{The work was partially supported  was supported by Agencia Estatal de Investigaci\'on (Spain), grant MTM2016-
        79661-P (European FEDER support included, UE). The third named author was supported by VI PPIT-US and by Instituto de Matem\'{a}ticas de la Universidad de Sevilla. }

    \begin{abstract}
We describe infinite-dimensional Leibniz algebras whose associated Lie algebra is the Witt algebra and we prove the triviality of low-dimensional Leibniz cohomology groups of the Witt algebra with the coefficients in itself.
 \end{abstract}

\subjclass[2010]{17A32, 17B30, 17B10}
\keywords{Leibniz algebra, Witt Lie algebra, Leibniz representation, right Lie module, classification}

\maketitle

\section{Introduction}

Mathematical structures are important mostly in mathematics and its applications. Leibniz algebras were introduced and developed by J-L. Loday \cite{Loday}. They are a generalization of Lie algebras, removing the restriction that the product is anti-commutative or that the square of an element is zero. Many results of Lie algebra theory are extended to the case of Leibniz algebras, while "pure" Leibniz results (which are true for non Lie Leibiz algebras) are obtained as well \cite{Albeverio, Ayupov2, Ayupov3, simple, Omirov3, Ayupov1, BarnesEngel, Barnes, Bosko, diamond, Omirov1, OmirovCartan, Omirov2}.

For Lie algebras it is known that an arbitrary finite-dimensional Lie algebra over a field of characteristic zero decomposes into the semidirect sum of the maximal solvable ideal and its semisimple subalgebra (Levi's Theorem, \cite{Jacob}).
Similar result is also true for Leibiz algebras, namely, a finite-dimensional Leibniz algebra decomposes into the semidirect sum of a maximal solvable ideal and a semisimple Lie subalgebra (Levi's Theorem, \cite{Barnes}). Complete description of semisimple finite-dimensional Lie algebras over a field of characteristic zero is known \cite{Hump, Jacob}. Therefore, the study of finite-dimensional Leibniz algebras is reduced to the study of solvable ones.

Infinite-dimensional case is more complicated even in Lie algebra structures. The most simple infinite-dimensional structure are Witt and Virasoro algebras. Complex Witt algebra was first considered by E. Cartan \cite{cartan} in 1909. This algebra is an example of an infinite-dimensional simple Lie algebra. Recall that Virasoro algebra was introduced due to Witt algebra. In fact, the Virasoro algebra is the one-dimensional central extension of the Witt algebra. In the work \cite{KacV} the infinite-dimensional Lie algebras and their representations are studied. We also mention that the investigations of representations of Witt algebra are found in \cite{WittRepr}, \cite{Zhao2}, \cite{Zhao1}.

A non-Lie Leibniz algebra $L$ contains the non-trivial ideal, generated by the squares of elements of the algebra $L$ (denoted by $\mathcal{I}$ and usually called the {\it Leibniz kernel}), i.e., $\mathcal{I}= \langle [x,x] \mid x \in L\rangle $. The ideal $\mathcal{I}$ plays an important role in the theory of Leibniz algebras since it determines the (possible) non-Lie property of a Leibniz algebra. Moreover, this ideal belongs to the right annihilator of $L$ and it is the minimal ideal with the property that the quotient algebra $L/\mathcal{I}$ is a Lie algebra.

The usual notion of simplicity for non Lie Leibniz algebras has no sense (because of non-triviality of the ideal $\mathcal{I}$). Therefore, it is proposed to use the adapted version of simplicity for Leibniz algebras. Namely, Leibniz algebra $L$ is called \emph{simple}, if $[L,L] \neq \mathcal{I}$ and its only ideals are \{0\}, $\mathcal{I}$ and $L$. Clearly, if a Leibniz algebra is simple, then its corresponding Lie algebra is simple as well. However, the converse is not true, in general. The analogue of Levi's theorem for Leibniz algebras imply that any finite-dimensional simple Leibniz algebra decomposes into the semidirect sum of simple Lie algebra and the ideal $\mathcal{I}$, where $\mathcal{I}$ can be considered as an irreducible right module over simple Lie algebra. The approach in the investigation of Leibniz algebras with a given corresponding solvable Lie algebra was applied in \cite{Omirov1, Omirov2, diamond}.

In finite-dimensional case we know that for a given simple Lie algebra and its irreducible right module one can construct a unique simple Leibniz algebra such that its corresponding Lie algebra is the simple Lie algebra and the ideal $\mathcal{I}$ is the given right module. In this paper we describe the Leibniz algebras whose corresponding Lie algebra is the Witt algebra and the ideal $\mathcal{I}$ is its right module (see Theorems \ref{thm11} and \ref{thm22}). Moreover, we prove the triviality of low-dimensional cohomology groups of these Leibniz algebras (see Theorem \ref{thm33}).

Throughout the paper algebras are considered to be over the field of the complex numbers. Moreover, in the table of multiplications of algebras the omitted products are assumed to be zero.

\section{Prelimiaries}
In this section we give preliminary definitions and results on Leibniz algebras and modules over the Witt algebra.

An algebra $L$ with multiplication $[\cdot,\cdot]$ over a field $F$ is called {\it Leibniz algebra} if for any $x,y,z\in L$ the so-called Leibniz identity
    $$[x,[y,z]]=[[x,y],z]-[[x,z],y]$$
holds.

Further we will use the notation
\[ {\mathcal L}(x, y, z)=[x,[y,z]] - [[x,y],z] + [[x,z],y].\]
It is obvious that Leibniz algebras are determined by the identity
${\mathcal L}(x, y, z)=0$.

For a Leibniz algebra $L$ we consider the natural homomorphism onto the quotient Lie algebra $L/\mathcal{I},$ which is called the \textit{corresponding Lie algebra} to the Leibniz algebra $L$ (in some references this algebra is called the \textit{liezation} of $L$).

Now we present a construction of Leibniz algebras by a given Lie algebra and its right module.

Let $(G, [-,-])$ be a Lie algebra and let $V$ be a right $G$-module. We equip the vector space $Q(G,V)=G\oplus V$ with the multiplication $(-,-)$ in the following way:
\begin{equation}\label{eq1}
(x+v,y+w) \coloneqq [x,y] + v \star y,  \quad \quad  x, y \in G, \ v,w \in \mathcal{V}.
\end{equation}

Then $Q(G,V)$ is a Leibniz algebra and from (\ref{eq1}) we get $(G,V) =(V,V) = 0$.

Note that if $V\star G=V$, then $V$ is nothing else but the ideal $\mathcal{I}$ of the Leibniz algebra $Q(G,V).$

The map $\mathcal{I}\times L/\mathcal{I} \longrightarrow \mathcal{I}$ defined as  $(v,\overline{x})\longrightarrow [v,x],$ $v\in \mathcal{I},$ $x\in L$ endows $\mathcal{I}$ with a structure of a right $L/\mathcal{I}$-module (it is well-defined due to $\mathcal{I}$ being in the right annihilator). Thus, for a given Leibniz algebra $L$ we have a Lie algebra $L/\mathcal{I}$ and its right module $\mathcal{I}$.

The main goal of this paper is to describe Leibniz algebras such that its corresponding Lie algebra is a given Lie algebra $G$ and the ideal $\mathcal{I}$ is a given right $G$-module $V$.

Let $A=\mathbb{C}[z,z^{-1}]$ be the algebra of Laurent polynomials in one variable. The Lie algebra of derivations $$Der(A)=span\{f(z)\frac{d}{dz}:f\in\mathbb{C}[z,z^{-1}]\}$$ is called \emph{Witt algebra} and denoted by $\mathcal{W}$. Obviously, the basis of $\mathcal{W}$ can be chosen as $\{d_j \ | \ j\in\mathbb{Z}\},$ where $d_j=-z^{j+1}\frac{d}{dz}.$

Then the table of multiplications of $\mathcal{W}$ in this basis have the following form:
\begin{equation}\label{equation1}
[d_m,d_n]=(m-n)d_{m+n}, \quad m,n\in \mathbb{Z}.
\end{equation}



On a vector space $V(\alpha, \beta)=\{v(n) \ | \ n\in \mathbb{Z}\}$ we define $\mathcal{W}$-module structure \cite{WittRepr} as follows:
\begin{equation}\label{eq2}
v(n)\star d_m=(\alpha+n+\beta m)v(n+m), \quad n\in \mathbb{Z}, \alpha, \beta\in \mathbb{C}, \ \alpha \neq0.
\end{equation}
This bracket would be discovered by the action $d_m=-z^{m+1}\frac{d}{dz}$ on the elements $v(n)=z^{n+\alpha}(dz)^\beta.$

Below we present a result of \cite{KacV}.

\begin{teo} The representation $V(\alpha,\beta)$ is reducible if $(i)$ $\alpha\in\mathbb{Z}$ and
$\beta=0,$ or $(ii)$ $\alpha\in\mathbb{Z}$ and $\beta=1;$ otherwise it is irreducible.
\end{teo}

Let $L$ be Leibniz algebra such that its corresponding Lie algebra is $\mathcal{W}$ and the right $\mathcal{W}$-module $V(\alpha, \beta)$ is the ideal $\mathcal{I}$. We equip the direct sum of the vector spaces $L=\mathcal{W}\oplus \mathcal{I}$ with a product $[-,-]$ satisfying the conditions:
\begin{equation}\label{eq3}
[V(\alpha, \beta),\mathcal{W}]=V(\alpha, \beta)\star\mathcal{W}, \quad [\mathcal{W},V(\alpha, \beta)]=[V(\alpha, \beta),V(\alpha, \beta)]=0,
\end{equation}
where the product $V(\alpha, \beta)\star\mathcal{W}$ follows from (\ref{eq2}).

Thus, in order to completely describe the Leibniz algebra $L$ we have to clarify the product $[\mathcal{W},\mathcal{W}]$.

\

For getting acquainted with cohomology of Leibniz algebras we refer the reader to works \cite{Loday}, \cite{Loday2} and references therein. Here we just give the definition of the second group of cohomology for Leibniz algebras with coefficient in itself. In fact, the \emph{second cohomology group} of a Leibniz algebra $L$ with coefficient itself is the quotient space
$$HL^2(L,L):=ZL^2(L,L)/BL^2(L,L),$$
where elements $f\in BL^2(L,L)$ and $\varphi \in ZL^2(L,L)$ are defined by conditions:
$$f(x,y)=[d(x),y] + [x,d(y)] - d([x,y]), \,\, \mbox{for some linear map} \,\, d\in Hom(L,L),$$
\begin{equation}\label{eq4}
[x,\varphi(y,z)] - [\varphi(x,y), z] +[\varphi(x,z), y] +\varphi(x,[y,z]) - \varphi([x,y],z)+\varphi([x,z],y)=0,
\end{equation}
respectively.

It is obvious that a Leibniz $2$-cocycle $\varphi$ of a Leibniz algebra $L$ is  determined by the identity
$\Phi(\varphi)(x, y, z)=0, \ x,y,z\in L$ where
$$\Phi(\varphi)(x, y, z)=[x,\varphi(y,z)] - [\varphi(x,y), z] +[\varphi(x,z), y] +\varphi(x,[y,z]) - \varphi([x,y],z)+\varphi([x,z],y).$$

\section{Some Leibniz algebras with corresponding Lie algebra $\mathcal{W}$.}

\

In this section we describe Leibniz algebras whose corresponding Lie algebra is the Witt algebra and the ideal $\mathcal{I}$ is $V(\alpha, \beta)$ which satisfy (\ref{eq3}). First, let us consider the case $\alpha\notin \mathbb{Z}$.

\begin{teo} \label{thm11} Let $L$ be an arbitrary Leibniz algebra whose corresponding Lie algebra is the Witt algebra $\mathcal{W}$ and the ideal $\mathcal{I}$
of $L$ is considered as $V(\alpha, \beta)$ which satisfy the condition (\ref{eq3}). If $\alpha\notin\mathbb{Z},$ then there exists a basis  $\{d_i, v(i) \ | \ i\in \mathbb{Z}\}$ of $L$ such that its table of multiplications have the form:
    $$\left\{\begin{array}{l}
        [v(k),d_i]=(k+\alpha+\beta i)v(i+k),\\[1mm]
        [d_i,d_j]=(i-j)d_{i+j}.
        \end{array}\right.$$
    \end{teo}
\begin{proof} Let $\{d_i \ | \ i\in \mathbb{Z}\}$ be the basis of $\mathcal{W}$ and let $\{v(i) \ | \ i\in \mathbb{Z}\}$ be the basis of the space $V(\alpha, \beta)$.

    Let us introduce notations
    $$\begin{array}{ll}
    [v(k),d_i]=(k+\alpha+\beta i)v(i+k),&k+\alpha+\beta i\notin\mathbb{Z}\\[1mm]
    [d_i,d_j]=(i-j)d_{i+j}+\sum\limits_{k}\gamma_{ijk}v(k), & i\notin \{j,-j\},\\[2mm]
    [d_j,d_{-j}]=2jd_0+\sum\limits_k\delta_{jk}v(k), & j\neq0,\\[2mm]
    [d_j,d_j]=\sum\limits_k\sigma_{jk}v(k),&
    \end{array}$$
    with $\alpha+j+\beta j\notin\mathbb{Z},$ $j\in\mathbb{Z}$  and $\alpha,\beta\in\mathbb{C}$.

    Taking into account $\alpha\notin\mathbb{Z}$ by taking the change of basis elements $d_i$ in the following way:
    \[d^\prime_0=d_0-\sum\limits_k\frac{\sigma_{0,k}}{\alpha+k}v(k), \quad
    d^\prime_s=d_s+\sum\limits_{k}\frac{\gamma_{s,0,k}}{s-\alpha-k}v(k), \ s\in \mathbb{Z}^*,\]
one can assume that
    \[ [d_0,d_0]=0, \, [d_j,d_0]=jd_j, \quad  j\in\mathbb{Z}.\]

    By considering Leibniz identity for triples of elements given below we have the following constraints:
    \[\begin{array}{lllll}
    \text{ Leibniz identity }& & \text{ Constraints } &\\[1mm]
    \hline \hline\\
    {\mathcal L}(d_0,d_j,d_0)=0&\Longrightarrow &\gamma_{0,j,k}=0,\ \alpha+k \neq j&\Longrightarrow[d_0,d_j]=-jd_j,& j\in\mathbb{Z},\\
    {\mathcal L}(d_j,d_j,d_0)=0&\Longrightarrow & \sigma_{j,k}=0,\,\, k\neq 2j-\alpha&\Longrightarrow[d_j,d_j]=0,&\\
    {\mathcal L}(d_j,d_0,d_{-j})=0&\Longrightarrow & \delta_{j,k}=0,\, k\neq -\alpha&\Longrightarrow[d_j,d_{-j}]=2jd_0,&\\
    {\mathcal L}(d_j,d_0,d_{-j})=0&\Longrightarrow & \gamma_{ijk}=0, \,\, k\neq i+j-\alpha&\Longrightarrow[d_i,d_j]=(i-j)d_{i+j}.&
    \end{array}\]

     Thus, we have proved $[d_i,d_j]=(i-j)d_{i+j}$, that is, $[\mathcal{W},\mathcal{W}]=\mathcal{W}$, which completes the description of the structure of $L$.
\end{proof}

Now we consider the case $\alpha\in \mathbb{Z}$.

\begin{teo} \label{thm22} Let $L$  be an arbitrary Leibniz algebra with corresponding Lie algebra $\mathcal{W}$ and the ideal $\mathcal{I}$ of $L$ is considered
as $V(\alpha, \beta)$ which satisfies the conditions (\ref{eq3}).
If $\alpha\in \mathbb{Z},$ then there exists a basis $\{d_i, v(i) \ | \ i\in \mathbb{Z}\}$ of
$L$ such that the table of multiplications of $L$ have one of the following form:
$$
\begin{array}{lll}
(I):\left\{\begin{array}{ll}
    [v(k),d_i]=(k+\alpha+\beta i)v(i+k),&  \beta\notin\{-1,0,1,2,3\},\\[1mm]
    [d_i,d_j]=(i-j)d_{i+j},&
\end{array}\right.& \\ [5mm]
(II):\left\{\begin{array}{ll}
    [v(k),d_i]=(k+\alpha+3i)v(i+k),&\\[1mm]
    [d_j,d_j]=b_{j,j}v(2j-\alpha), & j\notin\{-1,
    0,1\}, \\[1mm]
    [d_j,d_{-j}]=2jd_0-\frac{1}{3}b_{j,j}v(-\alpha), & j\notin\{-1,0,1\}, \\[1mm]
    [d_i,d_j]=(i-j)d_{i+j}+b_{i,j}v(i+j-1),&\ \ i\notin \{j, -j\},\\[1mm]
\end{array}\right.& \\ [5mm]
where & \\ [2mm]
\left\{\small\begin{array}{ll}
    b_{2,2}=9, \, b_{-2,-2}=-9,& \\[2mm]
    b_{i,i}=\displaystyle\frac{(i+1)(2i+1)b_{i-1,i-1}-(i+1)(4i-3)a_{i-1}-(4i-1)(i-2)a_{-i}}{(i-2)(2i-3)},\, i\notin\{-2,-1,0,1,2\},\\ [2mm]
    b_{i,i+1}=\displaystyle\frac{(i+1)((2i+3)b_{i,i}+(4i+1)a_i)}{(i-1)(2i+1)}, \quad \quad \quad \quad \quad \quad \quad \quad \quad \quad \quad \quad \quad i\notin\{-1,0,1\},& \\[2mm]
    b_{i,j}=\displaystyle\frac{j((i-1)b_{i+1,j-1}+(3i+j)(j-1)a_{j-1}-(j-i-1)(j+i-1)a_{j+i-1}-i(i+3j-2)a_i)}{i(j-2)}, \\[2mm]
    \qquad \quad \quad \quad \quad \quad \quad \quad \quad \quad \quad \quad \quad \quad \quad \quad \quad \quad \quad \quad \quad \quad   j\neq i+1,\ i\notin \{j,-j\}, \, i,j \notin\{-1,0,1\},\\[1mm]
    a_i=\displaystyle\frac{(i-1)(i+2)(i+3)}{20}, \, a_{-i}=-\displaystyle\frac{(i+1)(i-2)(i-3)}{20}, \quad \quad \quad \quad \quad \quad \quad  i\notin\{-1,0,1\}, &\\ [1mm]
\end{array}\right. & \\ [15mm]
(III): \left\{\begin{array}{ll}
    [v(k),d_i]=(k+\alpha+i)v(i+k),&\\[1mm]
    [d_0,d_i]=-id_i,&\ \ i\neq 0,\\[1mm]
    [d_i,d_0]=id_i,&\ \ i\neq 0,\\[1mm]
    [d_i,d_i]=(i^3-i) v(2i-\alpha),& \ \ i\neq 0,\\[1mm]
    [d_i,d_{-i}]=2id_0+(i^3-i) v(-\alpha),&\ \ i\neq 0,\\[1mm]
    [d_i,d_j]=(i-j)d_{i+j}+j(ij-1) v(i+j-\alpha),&\ \ i\notin \{j,-j,0\}, \ \ j\neq 0.\\[1mm]
\end{array}\right. & \\ [15mm]
(IV):\left\{\begin{array}{ll}
    [v(k),d_i]=(\alpha+k-i)v(k+i),\\[1mm]
    [d_j,d_j]=b_{j,j}v(2j-\alpha),& \ \ j\neq 0,\\[1mm]
    [d_j,d_{-j}]=2jd_0-\frac{1}{3}b_{j,j}v(-\alpha), &\ \ j\neq 0,\\[1mm]
    [d_i,d_j]=(i-j)d_{i+j}+b_{i,j}v(i+j-\alpha),&\ \ i\notin \{j,-j\},\\[1mm]
\end{array}\right.& \\ [5mm]
where & \\ [2mm]
\left \{ \begin{array}{ll}
    b_{2,2}=2, \, b_{-2,-2}=-2,& \\[1mm]
    b_{i,i}=\frac{1}{(2i+1)(i-2)}((i+1)(2i-3)b_{i-1,i-1}-(2i-1)), & i\notin\{-2,-1,0,1,2\},\\ [1mm]
    b_{i,i+1}=\frac{i+1}{(2i+1)(i-1)}((2i-1)b_{i,i}-1),& \\ [1mm]
    b_{i,j}=\frac{j}{i(j-2)}((i-1)b_{i+1,j-1}+(j-i-1)), \, &i\notin \{j,-j,\},  \, i,j \notin\{-1,0,1\},\ j\neq i+1.
    \end{array}\right.
\end{array}
$$
\end{teo}
\begin{proof} Let us chose a basis $\{d_i, v(i) \ | \ i\in \mathbb{Z}\}$ such that $\{d_i, | \ i\in \mathbb{Z}\}$ and $\{v(n) \ | \ n\in \mathbb{Z}\}$ are bases of the spaces $\mathcal{W}$ and  $V(\alpha, \beta)$, respectively. From the assumption of theorem we have the products
$$[v(n),d_m]=(\alpha+n+\beta m)v(n+m), \  n\in \mathbb{Z}, \alpha, \beta\in \mathbb{C}, \ \alpha \neq0,$$
$$[\mathcal{W}, V(\alpha, \beta)]=[V(\alpha, \beta), V(\alpha, \beta)]=0.$$

In order to complete the proof of theorem, that is, to clarify the structure of $L$ we have to describe the products $[\mathcal{W},\mathcal{W}]$.

We introduce notations:
$$\begin{array}{ll}
[d_i,d_j]=(i-j)d_{i+j}+\sum\limits_{k}\gamma_{i,j,k}v(k), & i\neq-j,\ \quad i\neq j,\\{}
[d_j,d_{-j}]=2jd_0+\sum\limits_k\gamma_{j,-j,k}v(k), & j\neq0,\\{}
[d_j,d_j]=\sum\limits_k\gamma_{j,j,k}v(k).&
\end{array}$$

Taking the change of basis elements $d_i$ as follows:
$$d^\prime_0=d_0-\sum\limits_{k,k\neq -\alpha}\frac{\gamma_{0,0,k}}{\alpha+k}v(k),
\ \ d^\prime_s=d_s+\sum\limits_{k,k\neq
s-\alpha}\frac{\gamma_{s,0,k}}{s-\alpha-k}v(k)\ \quad s\neq 0.$$
one can assume that
$$\begin{array}{l}
[d_0,d_0]=\gamma_{0,0,-\alpha}v(-\alpha),\\{}
[d_j,d_0]=jd_j+\gamma_{j,0,j-\alpha}v(j-\alpha),\ \ j\neq0.
\end{array}$$

Considering the Leibniz identity for the triples of elements given below we have the following constraints:
\[\begin{array}{llll}
\text{ Leibniz identity }& & \text{ Constraints } &\\[1mm]
\hline \hline\\
{\mathcal L}(d_0,d_0,d_j)=0,\ j\neq 0&\Longrightarrow &
$$\left\{\begin{array}{ll}
\gamma_{j,0,j-\alpha}=-\beta\gamma_{0,0,-\alpha}, & j\neq 0,\\
\gamma_{0,j,k}=0, & k\neq j-\alpha,\\{}
[d_0,d_j]=-jd_j+\gamma_{0,j,j-\alpha}v(j-\alpha), & j\neq0.
\end{array}\right.$$
\\[10mm]
{\mathcal L}(d_j,d_j,d_0)=0,\ j\neq 0&\Longrightarrow &
$$\left\{\begin{array}{ll}
(1+\beta)\gamma_{j,0,j-\alpha}=0, & j\neq0,\\
\gamma_{j,j,k}=0, & k\neq 2j-\alpha,\\{}
[d_j,d_j]=\gamma_{j,j,2j-\alpha}v(2j-\alpha), & j\in\mathbb{Z},
\end{array}\right.$$\\[10mm]
{\mathcal L}(d_j,d_0,d_{-j})=0,\ j\neq 0&\Longrightarrow &
$$\left\{\begin{array}{ll}
\gamma_{j,-j,k}=0, \quad \quad  \quad  \quad  \quad   k\neq -\alpha,\\
(1-\beta)\gamma_{j,0,j-\alpha}-2\gamma_{0,0,-\alpha}=0,& \\{}
[d_j,d_{-j}]=2jd_0+\gamma_{j,-j,-\alpha}v(-\alpha).
\end{array}\right.$$
\end{array}\]

For the sake of convenience we denote
$$\gamma_{0,0}=\gamma_{0,0,-\alpha}, \quad \gamma_{0,j}=\gamma_{0,j,j-\alpha}, \quad \gamma_{j,j}=\gamma_{j,j,2j-\alpha}, \quad \gamma_{j,-j}=\gamma_{j,-j,-\alpha}.$$


Making the change
$$ d_j'=d_j-\frac{\gamma_{0,j}}{j}v(j-\alpha), \ \ j\neq0,$$
we derive $\gamma_{0,j}=0.$

The equalities ${\mathcal L}(d_i,d_0,d_j)=0$ with $i\neq \{-j, 0, j\}$ imply
$\gamma_{i,j,k}=0,  \ k\neq i+j-\alpha,   \ i\notin \{-j, 0,j\}.$ Therefore, we apply new notations $\gamma_{i,j}=\gamma_{i,j,i+j-\alpha}$.

Considering the Leibniz identity for the triples elements given below we obtain the constraints:
\[\begin{array}{llll}
\text{ Identity }& & \text{ Constraints } &\\[1mm]
\hline \hline\\
\mathcal{L}(d_0,d_j,d_{-j})=0 &\Longrightarrow &
\gamma_{j,-j}+\gamma_{-j,j}=-2\gamma_{0,0},\ j\neq 0
\\[5mm]
\mathcal{L}(d_0,d_i,d_j)=0 &\Longrightarrow &
i\gamma_{i,j}=j\gamma_{j,i},\ \ i\neq \{0,j,-j\},\ j\neq 0\\[5mm]
\mathcal{L}(d_j,d_{-j},d_{j})=0 &\Longrightarrow &
\beta\gamma_{j,-j}+(\beta-2)\gamma_{j,j}=2\beta \gamma_{0,0},\ j\neq 0,\\[5mm]
\mathcal{L}(d_i,d_j,d_k)=0 &\Longrightarrow &
(j-k)\gamma_{i,j+k}=(i-j)\gamma_{i+j,k}+(i+j+\beta k)\gamma_{i,j}-(i-k)\gamma_{i+k,j}-\\[5mm]
&&(i+k+\beta j)\gamma_{i,k}.\end{array}\]

In new notations the table of multiplication of $L$ has the form
$$\left\{\begin{array}{ll}
[v(k),d_i]=(k+\alpha+\beta i)v(i+k),&\\[1mm]
[d_0,d_0]=\gamma_{0,0}v(-\alpha),&\\[1mm]
[d_0,d_j]=-jd_j,&\ \ j\neq 0,\\[1mm]
[d_j,d_0]=jd_j+\gamma_{j,0}v(j-\alpha),&\ \ j\neq 0,\\[1mm]
[d_j,d_j]=\gamma_{j,j}v(2j-\alpha),& \ \ j\neq 0,\\[1mm]
[d_j,d_{-j}]=2jd_0+\gamma_{j,-j}v(-\alpha),&\ \ j\neq 0,\\[1mm]
[d_i,d_j]=(i-j)d_{i+j}+\gamma_{i,j}v(i+j-\alpha),&\ \ i\notin \{-j,0, j\}.\\[1mm]
\end{array}\right. $$
with the following restrictions:
$$\begin{array}{l}
\qquad \gamma_{j,0}+\beta\gamma_{0,0}=(1+\beta)\gamma_{j,0}=(1-\beta)\gamma_{j,0}-2\gamma_{0,0}=0,\ j\neq 0
\end{array}$$
\begin{equation}{\label{eq5}}
\gamma_{j,-j}+\gamma_{-j,j}=-2\gamma_{0,0},\quad j\neq 0,
\end{equation}
\begin{equation}{\label{eq6}}
i\gamma_{i,j}=j\gamma_{j,i},\quad i\neq \{0,j,-j\},\ j\neq 0,
\end{equation}
\begin{equation}{\label{eq7}}
\beta\gamma_{j,-j}+(\beta-2)\gamma_{j,j}=2\beta\gamma_{0,0},\quad j\neq 0,
\end{equation}
\begin{equation}{\label{eq8}}
(j-k)\gamma_{i,j+k}=(i-j)\gamma_{i+j,k}+(i+j+\beta k)\gamma_{i,j}-(i-k)\gamma_{i+k,j}-(i+k+\beta j)\gamma_{i,k}.
\end{equation}

Let us consider possible cases.

\noindent $\bullet$ \textbf{Case 1.} Let $\beta\neq -1$. Then $\gamma_{0,0}=\gamma_{j,0}=0.$ 

\textbf{Case 1.1.} Let $\beta\neq 0$. Taking the following basis transformation
$$d_0'=d_0+\frac{\gamma_{1,1}}{\beta(\beta+1)}v(-\alpha), \, \, d_i'=d_i-\frac{\gamma_{1,1}}{1+\beta}v(i-\alpha),\ \ i\neq0, \ v'(i)=v(i),$$
we may assume $\gamma_{1,1}=0.$ From the above restrictions we derive $\gamma_{1,-1}=\gamma_{-1,1}=0.$

Equality (\ref{eq8}) with the values $i=k=1$ imply
$$(j-1)(\gamma_{1,j+1}+\gamma_{j+1,1})=(j+1+\beta)\gamma_{1,j}.$$

From this equality and (\ref{eq6}) we obtain
\begin{equation}{\label{eq9}}
(i-1)(i+2)\gamma_{i+1,1}=(i+1+\beta)i\gamma_{i,1},  \ i\notin
\{-2,-1,0,1\}.
\end{equation}
Using the induction it is easy to proof the equality
$$\gamma_{i,1}=\frac{(i-1)(3+\beta)(4+\beta)\cdots (i+\beta)}{4\cdot 5\cdots (i+1)}\gamma_{2,1},\quad i\in \mathbb{Z}^{+}\setminus \{1,2\}$$

Equality (\ref{eq8}) with the values $i=k=-1$ imply
$$(j+1)(\gamma_{-1,j-1}+\gamma_{j-1,-1})=(j-1-\beta)\gamma_{-1,j}-(\beta j-2)\gamma_{-1,-1}.$$
From this equality together with (\ref{eq6}) we get
\begin{equation}{\label{eq10}}
(i+1)(i-2)\gamma_{i-1,-1}=(i-1-\beta)i\gamma_{i,-1}-(\beta
i-2)\gamma_{-1,-1},  \ i \notin \{-1,0,1,2\}.
\end{equation}

\textbf{Case 1.1.1.} Let $\beta\neq 1$.

\textbf{Case 1.1.1.1.} Let $\beta\neq 2$. Then equality (\ref{eq7}) with $j=-1$ imply $\gamma_{-1,-1}=0.$

Substituting instead of parameters $\{i,j,k\}$ the following values $$(1,-2,1), (1,-2,-1), (-1,2,-1), (-1,2,1), (1,2,-1), (-1,-2,1)$$ in equality (\ref{eq8}) we obtain
$$\gamma_{1,-2}=\gamma_{1,-3}=\gamma_{-1,2}=\gamma_{-1,3}=(3-\beta)\gamma_{1,2}=(\beta-3)\gamma_{-1,-2}=0.$$%
%
%

From these equalities and (\ref{eq6}) we get $\gamma_{-2,1}=\gamma_{-3,1}=\gamma_{2,-1}=\gamma_{3,-1}=0.$

We distinguish the possible subcases.

\textbf{(a)} Let $\beta\neq3.$ Then $\gamma_{1,2}=\gamma_{-1,-2}=\gamma_{-2,-1}=\gamma_{2,1}=0$. From the equalities (\ref{eq9}), (\ref{eq10}) and (\ref{eq6}) we derive
$$\gamma_{i,1}=\gamma_{1,i}=\gamma_{-i,-1}= \gamma_{-1,-i}=0, \ i\in\mathbb{Z}^+\setminus\{1,2\}.$$

Equality (\ref{eq8}) with $i=\pm 1, j:=\pm i, k=\mp 1$ together with (\ref{eq6}) imply
$$\gamma_{i,-1}=\gamma_{-1,i}=\gamma_{-i,1}=\gamma_{1,-i}=0, \quad i\in \mathbb{Z}^+\setminus \{1,2\}.$$

Let us write the combining of the above restriction
\begin{equation}\label{eq1111}
  \gamma_{1,i}=\gamma_{i,1}=\gamma_{-1,i}=\gamma_{i,-1}=0,\quad i\in\mathbb{Z}^*.
\end{equation}

Substituting in (\ref{eq8}) $i=1,$ $j:=i, k:=j$ we get
$$(1-i)\gamma_{1+i,j}=(1-j)\gamma_{1+j,i}, \quad i,j\in \mathbb{Z}\setminus\{-1,1\}.$$
In this equality putting $j=-2$ we obtain $(1-i)\gamma_{1+i,-2}=3\gamma_{-1,i}=0$.

Therefore,  $\gamma_{i,-2}=\gamma_{-2,i}=0$ for $i\in \mathbb{Z}\setminus\{-2,0,2\}.$

Similarly, for $j\leq -3$ we deduce $\gamma_{i,j}=\gamma_{j,i}=0$ with $i\in \mathbb{Z}^*.$ So, we have
$\gamma_{i,j}=\gamma_{j,i}=0$ with $i\in \mathbb{Z}^*$ and $j\in \mathbb{Z}^{-}$.

From equality (\ref{eq8}) we have
\begin{equation}\label{eq88}
(i+(1-\beta)j)\gamma_{i,j}=(i+j)\gamma_{i-j,j}, \ i\neq j, i,j\in\mathbb{Z}^{+}.
\end{equation}

 If $i+(1-\beta)j\neq 0$, then from (\ref{eq88}) with $i=2, \ j\geq3$ we obtain $\gamma_{2,j}=\gamma_{j,2}=0,$ $j\in\mathbb{Z}^{+}\setminus\{1,2\}$ and repeating this consideration for $i\geq3$, we deduce
$$\gamma_{i,j}=\gamma_{j,i}=0, \ i\notin\{j,-j\}, \ i,j\in\mathbb{Z}^*.$$

If $i+(1-\beta)j=0$, then equality (\ref{eq88}) implies
$$\gamma_{i,j}=\gamma_{j,i}=0, i\notin\{j,-j\}, \ i,j\in\mathbb{Z}^*.$$

Taking into account the above equation and equality (\ref{eq8}) with $j:=i, \ k:=j$ and $i\notin\{-j\}$ we get
$\gamma_{i,i}=0$. Finally, from (\ref{eq7}) we have $\gamma_{i,-i}=0, \  i\in\mathbb{Z}^*$ and the first algebra of the list of theorem is obtained.

\

\textbf{(b)} Let $\beta=3.$ Then we have
$$\gamma_{1,-2}=\gamma_{-2,1}=\gamma_{-1,2}=\gamma_{2,-1}=\gamma_{1,-3}=\gamma_{-3,1}=\gamma_{-1,3}=\gamma_{3,-1}=
\gamma_{1,1}=\gamma_{1,-1}=\gamma_{-1,1}=0.$$
Note that from (\ref{eq7}) we get $\gamma_{j,j}=-3\gamma_{j,-j}.$

Substituting $\beta=3$ in equalities (\ref{eq9}) and (\ref{eq10}) we derive
\begin{equation}{\label{eq11}}
(i-1)(i+2)\gamma_{i+1,1}=(i+4)i\gamma_{i,1}, \quad i\in \mathbb{Z}\setminus \{-1,0,1,2\},
\end{equation}
\begin{equation}{\label{eq12}}
(1+i)(i-2)\gamma_{i-1,-1}=(i-4)i\gamma_{i,-1}, \quad i\in \mathbb{Z}\setminus \{-1,0,1,2\}.
\end{equation}

Applying induction in equality (\ref{eq11}) one can prove
\begin{equation}\label{eq222}
  \gamma_{i,1}=\frac{(i-1)(i+2)(i+3)}{20}\gamma_{2,1}, \ \ i\geq 3.
\end{equation}

Similarly, using (\ref{eq12}) it is easy to prove that
\begin{equation}\label{eq2222}
  \gamma_{i,-1}=-\frac{(i+1)(i-2)(i-3)}{20}\gamma_{-2,-1}, \ \ i\leq -3.
\end{equation}

Substituting instead of parameters $\{i,j,k\}$ the following values $(1,i,-1), \ (-1,-i,1)$ in equality (\ref{eq8}) we obtain
$$(1-i)\gamma_{i+1,-1}=(1+i)\gamma_{1,i-1}-(i-2)\gamma_{1,i}, \ i\in \mathbb{Z}\setminus \{1,2\},$$
$$(i-1)\gamma_{-i-1,1}=-(1+i)\gamma_{-1,-i+1}+(i-2)\gamma_{-1,-i}, \ i\in \mathbb{Z}\setminus \{1,2\}.$$

Now substitute (\ref{eq222}) in the first of the above equations to deduce
$$\gamma_{i,-1}=\frac{(i+1)(i-2)(i-3)}{20}\gamma_{2,1}, \ \ i\in \mathbb{Z}^+\setminus\{1,2\}.$$

Similarly, applying (\ref{eq2222}) in the second of the above equalities we get
$$\gamma_{-i,1}=\frac{(i+1)(i-2)(i-3)}{20}\gamma_{-2,-1}, \ \ i\in \mathbb{Z}^+\setminus\{1,2\}.$$

Taking into account that $\gamma_{2,-2}=-\gamma_{-2,2}$ in the equalities $\gamma_{2,2}=-3\gamma_{2,-2} \quad \gamma_{-2,-2}=-3\gamma_{-2,2}$ we obtain $\gamma_{2,2}+\gamma_{-2,-2}=0.$ Considering equality (\ref{eq8}) with values $(i,j,k)$ equal to $(2,2-1)$ and $(-2,-2,1)$ lead to
$\gamma_{2,2}=9\gamma_{2,1}, \ \gamma_{-2,-2}=9\gamma_{-2,-1}.$

Consequently, $\gamma_{2,1}=-\gamma_{-2,-1}$ and
$$\gamma_{i,1}=a_i\gamma_{2,1}, \quad \gamma_{i,-1}=-a_{-i}\gamma_{2,1}, \quad i\in\mathbb{Z}\setminus\{-1,0,1\},$$
where
$$a_i=\frac{(i-1)(i+2)(i+3)}{20}, \ \ a_{-i}=-\frac{(i+1)(i-2)(i-3)}{20}.$$

If in equality (\ref{eq8}) we take $(i,j,k)=(i,i,\mp 1)$, then we obtain

\begin{equation}{\label{eq13}} (i+1)(\gamma_{i,i-1}+\gamma_{i-1,i})=(2i-3)\gamma_{i,i}-(4i-1)\gamma_{i,-1}, \, i\notin\{-1,0,1\},
\end{equation}
\begin{equation}{\label{eq14}} (i-1)(\gamma_{i,i+1}+\gamma_{i+1,i})=(2i+3)\gamma_{i,i}-(4i+1)\gamma_{i,1}, \, i\notin\{-1,0,1\}
\end{equation}

Putting $i=2$ and $i=3$ in (\ref{eq14}) and (\ref{eq13}), respectively, we derive
$$\gamma_{2,3}+\gamma_{3,2}=54\gamma_{2,1}, \quad 4(\gamma_{2,3}+\gamma_{3,2})=3\gamma_{3,3}
\quad \Longrightarrow \gamma_{3,3}=72\gamma_{2,1}.$$

For $ i\notin \{-2,-1,0,1,2\}$ we set $\gamma_{i,i}=b_{i,i}\gamma_{2,1}.$ Clearly, $b_{2,2}=9$ and $ b_{3,3}=72.$

 If we replace the parameter $i$ in (\ref{eq14}) to $i-1$
 and multiply both sides of equality (\ref{eq13}) by $i-2,$ then we derive
$$\left\{\begin{array}{l}
(i-2)(\gamma_{i-1,i}+\gamma_{i,i-1})=(2i+1)\gamma_{i-1,i-1}-(4i-3)\gamma_{i-1,1},\\
(i+1)(i-2)(\gamma_{i-1,i}+\gamma_{i,i-1})=(i-2)(2i-3)\gamma_{i,i}-(i-2)(4i-1)\gamma_{i,-1}.
\end{array}\right.$$
From which we have
$$(i+1)((2i+1)\gamma_{i-1,i-1}-(4i-3)\gamma_{i-1,1})=(i-2)(2i-3)\gamma_{i,i}-(i-2)(4i-1)\gamma_{i,-1}.$$

Taking into account that $\gamma_{i-1,i-1}=b_{i-1,i-1}\gamma_{2,1}, \gamma_{i-1,1}=a_{i-1}\gamma_{2,1}, \gamma_{i,-1}=-a_{-i}\gamma_{2,1},$
we get $$\gamma_{i,i}=\frac{1}{(i-2)(2i-3)}((i+1)(2i+1)b_{i-1,i-1}-(i+1)(4i-3)a_{i-1}-(4i-1)(i-2)a_{-i})\gamma_{2,1}.$$

Thus, we obtain the following recursive relations:
$$\begin{array}{ll}
b_{2,2}=9, \ b_{-2,-2}=-9, & \\
b_{i,i}=\displaystyle\frac{(i+1)(2i+1)b_{i-1,i-1}-(i+1)(4i-3)a_{i-1}-(4i-1)(i-2)a_{-i}}{(i-2)(2i-3)}, & i\neq\{-2,-1,0,1,2\}.\\
\end{array}$$

Since $i\gamma_{i,j}=j\gamma_{j,i},$ in order to clarify the parameters $\gamma_{i,j}$ it is enough to discover $\gamma_{i,i+j}.$ First, we find parameters $\gamma_{i,i+1}.$ The analysis of (\ref{eq6}) and (\ref{eq14}) lead to
$$\gamma_{i,i+1}=b_{i,i+1}\gamma_{2,1}, \ \ i\neq\{-1,0,1\},$$
where $b_{i,i+1}=\displaystyle\frac{i+1}{(i-1)(2i+1)}((2i+3)b_{i,i}-(4i+1)a_i).$

Considering (\ref{eq8}) for $i=1, \ j:=i$ and $k:=i+1$ with $i\neq\{-1,0,1\}$ we have
$$\gamma_{i,i+2}=b_{i,i+2}\gamma_{2,1}, \quad i\neq\{-1,0,1\},$$
where $b_{i,i+2}=\displaystyle\frac{i+2}{i^2}((i-1)b_{i+1,i+1}+2(2i+1)(i+1)a_{i+1}-(2i+1)a_{2i+1}-i(4i+4)a_i).$

Applying similar arguments we obtain
$$\gamma_{i,j}=b_{i,j}\gamma_{2,1},$$
where $$\begin{array}{lll}
b_{i,j}&=&\displaystyle\frac{j}{i(j-2)}((i-1)b_{i+1,j-1}+(3i+j)(j-1)a_{j-1}-\\
&-&(j-i-1)(j+i-1)a_{j+i-1}-i(i+3j-2)a_i).
\end{array}$$

Finally, from (\ref{eq7}) we get $\gamma_{i,-i}=-\frac{1}{3}b_{i,i}\gamma_{2,1},$

If $\gamma_{2,1}=0,$ then we have the algebra $(I)$ with $\beta=3.$

If $\gamma_{2,1}\neq 0,$ using the scale of basis we may assume $\gamma_{2,1}=1$. Thus, the second algebra of the list of theorem is obtained.


\textbf{Case 1.1.1.2.} Let $\beta=2$. Then, due to (\ref{eq7}) we get $\gamma_{-j,j}=\gamma_{j,-j}=0.$ Note that $\gamma_{1,1}=\gamma_{-1,1}=\gamma_{1,-1}=0.$ Consequently, the table of multiplications of the algebra has the following form:
$$\left\{\begin{array}{ll}
[v(k),d_i]=(k+\alpha+\beta i)v(i+k),&\\[1mm]
[d_0,d_j]=-jd_j,&\ \ j\neq 0,\\[1mm]
[d_j,d_0]=jd_j,&\ \ j\neq 0,\\[1mm]
[d_j,d_j]=\gamma_{j,j}v(2j-\alpha),& \ \ j\notin\{0,1\},\\[1mm]
[d_j,d_{-j}]=2jd_0,&\ \ j\neq 0,\\[1mm]
[d_i,d_j]=(i-j)d_{i+j}+\gamma_{i,j}v(i+j-\alpha),&\ \ i\neq j,\ \ i\neq -j.\\[1mm]
\end{array}\right.$$

Substituting instead of parameters $\{i,j,k\}$ the following values $$(i,i,-2i),
(1,2,1), (1,-2,1), (1,-2,-1), (-1,2,1), (1,2,-1)$$ in equality (\ref{eq8}) we obtain restriction on structure constants of the algebra
\begin{equation}{\label{eq15}}
    \begin{array}{ll}
        2\gamma_{i,i}=-\gamma_{i,-2i}, \quad i\neq 0, \quad \gamma_{1,3}=-\gamma_{3,1}+5\gamma_{1,2}, \quad \gamma_{1,-2}=0, \\
        3\gamma_{-1,-1}=-\gamma_{1,-3}, \quad \gamma_{-1,3}=3\gamma_{-1,2}, \quad \gamma_{3,-1}=\gamma_{1,2}.
    \end{array}
\end{equation}

Taking into account (\ref{eq15}) and equation (\ref{eq8}) with $j:=i, k:=j$ we get
$$-(i+j)\gamma_{i,-2i}=(i-j)(\gamma_{i+j,i}+\gamma_{i,i+j})+(j+3i)\gamma_{i,j}.$$
This equality with $i=1$ has the form
$$-(1+j)\gamma_{1,-2}=(1-j)(\gamma_{1+j,1}+\gamma_{1,1+j})+(j+3)\gamma_{1,j}.$$
Due to (\ref{eq6}) the above equality can be written as follows
\begin{equation}{\label{eq16}}
    (1-j)(2+j)\gamma_{1+j,1}+j(j+3)\gamma_{j,1}=-(1+j)\gamma_{1,-2}.
\end{equation}

Equality (\ref{eq16}) imply
$$\begin{array}{lll}
4\gamma_{3,1}=5\gamma_{1,2},&\gamma_{-2,1}=0,&\gamma_{-1,3}=-3\gamma_{1,2}
\end{array}
$$

Considering (\ref{eq8}) with $(i, j, k)=(-1,3, -1)$ and (\ref{eq6}) we conclude
$$\gamma_{-1,1}=-\gamma_{-1,2}.$$
Thanks to (\ref{eq15}) we have $\gamma_{-1,-1}=\gamma_{-1,2}=0.$

From equality (\ref{eq8}) with $j:=i, \ k:=j$ and $i\notin\{-j\}$ we obtain
$$(i-j)(\gamma_{i,i+j}+\gamma_{i+j,i})=2(i+j)\gamma_{i,i}-(3i+j)\gamma_{i,j}.$$

Setting $i=-1$ in the above equality and using (\ref{eq6}), we obtain
$$(j+1)(j-2)\gamma_{j-1,-1}=(j-3)\gamma_{-1,j},$$
from which we derive $\gamma_{i,-1}=0$ with $i\in \mathbb{Z}^+\setminus \{1,2\}.$

Equality (\ref{eq16}) for $j\leq -4$ deduces $\gamma_{i,1}=\gamma_{1,i}=0, \ i\in\mathbb{Z^{-}}.$

Substituting in equality (\ref{eq8}) the values of parameters $(i,j,k)$ as follows: $$(1,i,-1), (-1,i,1),  (1,i,j),$$
we derive relations:
\begin{equation}{\label{eq17}}
    \gamma_{1,i}=(1+i)\gamma_{1,i-1}, \, i\in \mathbb{Z}^+\setminus\{1,2\},
\end{equation}
\begin{equation}{\label{eq18}}
    (i+1)\gamma_{-1,i}=(i-1)\gamma_{-1,i+1}, \, i\in \mathbb{Z}^-\setminus\{-1\},
\end{equation}
\begin{equation}{\label{eq19}}
    (1-i)\gamma_{1+i,j}=(1-j)\gamma_{1+j,i}, \, i\in \mathbb{Z}^-\setminus\{-1,1\}.
\end{equation}

Taking $i\geq 3$ in (\ref{eq17}), $i\leq -2$ in (\ref{eq18}) and $j=-2$ in (\ref{eq19}),
we obtain
$$\gamma_{1,i}=\gamma_{i,1}=0, \ i\in \mathbb{Z}, \quad \gamma_{-1,i}=\gamma_{i,-1}=0, \ i\in \mathbb{Z}, \quad \gamma_{i,-2}=\gamma_{-2,i}=0, \ i\in \mathbb{Z}\setminus\{0,2\}.$$

Applying similar arguments for $j=-3,-4,\dots$ we get
$$\gamma_{i,j}=\gamma_{j,i}=0, \ i\neq j, \ i\in \mathbb{Z}^*, \ j\in \mathbb{Z}^-.$$
Thanks to the first equality of (\ref{eq15}) we have $\gamma_{i,i}=0, i\in \mathbb{Z}.$

From  equality (\ref{eq8}) with $k:=-j$ and $i\neq j$
we derive
$$(i-j)\gamma_{i,j}=(i+j)\gamma_{i-j,j}, \ \ \ i\neq j, \ \ i,j\in \mathbb{Z}^{+}.$$

Considering in this equality various values of parameters $i$ and $j$ we deduce
$$\gamma_{i,j}=\gamma_{j,i}=0, \ \ i\neq j,-j,\ \ i,j\in\mathbb{Z}^*.$$ Thus, we obtain the algebra $(I)$ with $\beta=2.$

\

\textbf{Case 1.1.1.} Let $\beta=1$. Then from (\ref{eq5}) and (\ref{eq7}) we have $\gamma_{1,1}=\gamma_{-1,1}=\gamma_{1,-1}=\gamma_{-1,-1}=0$ and the following family of algebras:
$$\left\{\begin{array}{ll}
[v(k),d_i]=(k+\alpha+i)v(i+k),&\\[1mm]
[d_0,d_j]=-jd_j,&\ \ j\neq 0,\\[1mm]
[d_j,d_0]=jd_j,&\ \ j\neq 0,\\[1mm]
[d_j,d_j]=\gamma_{j,j}v(2j-\alpha),& \ \ j\neq 0,\\[1mm]
[d_j,d_{-j}]=2jd_0+\gamma_{j,j}v(-\alpha),&\ \ j\neq 0,\\[1mm]
[d_i,d_j]=(i-j)d_{i+j}+\gamma_{i,j}v(i+j-\alpha),&\ \ i\neq j,\ \ i\neq -j.\\[1mm]
\end{array}\right.$$

Substituting in  equality (\ref{eq8}) the values of parameters $(i,j,k)$ as follows:
$$(-1,2,1), (1,-2,-1), (1,2,-1), (-1,-2,1), (2,2,-1),$$
$$(-2,-2,1), (1,2,-2), (2,-2,1), (-1,-2,2), (2,-2,-1),$$
we get the following relations between structure constants:
$$\begin{array}{lll}
\gamma_{-3,1}=4\gamma_{-2,-1},& \gamma_{1,-3}=-12\gamma_{-2,-1}, & \gamma_{3,-1}=4\gamma_{2,1},\\
\gamma_{-1,3}=-12\gamma_{2,1},& \gamma_{1,-2}=-6\gamma_{-2,-1}, & \gamma_{-2,1}=3\gamma_{-2,-1},\\
\gamma_{-1,2}=-6\gamma_{2,1},& \gamma_{2,-1}=3\gamma_{2,1},&\gamma_{2,2}=6\gamma_{2,1},\\
\gamma_{-2,-2}=6\gamma_{-2,-1},&\gamma_{3,-2}=20\gamma_{2,1}+6\gamma_{-2,-1},&\gamma_{-2,3}=-30\gamma_{2,1}-9\gamma_{-2,-1},\\
\gamma_{2,-2}=12\gamma_{2,1}+6\gamma_{-2,-1},&\gamma_{-2,2}=-12\gamma_{2,1}-6\gamma_{-2,-1},&\gamma_{-3,2}=20\gamma_{-2,-1}+6\gamma_{2,1},\\ \gamma_{2,-3}=-30\gamma_{-2,-1}-9\gamma_{2,1},&\gamma_{2,1}=-\gamma_{-2,-1}.
\end{array}$$

Note that from (\ref{eq7}) with $j=2$ we have $\gamma_{2,-2}=\gamma_{2,2}.$

From the identity (\ref{eq9}) we obtain
$$\begin{array}{ll}
\gamma_{i,1}=(i-1)\gamma_{2,1},& i\in\mathbb{Z}^+\setminus\{1,2\}, \\
\gamma_{-i,-1}=(i-1)\gamma_{-2,-1}=-(i-1)\gamma_{2,1},& i\in\mathbb{Z}^+\setminus\{1,2\}.
\end{array}$$

Considering (\ref{eq8}) for $i=\pm1, \ j:=\pm i, \ k:=\mp1, j$ and applying (\ref{eq6}) we conclude
$$\gamma_{i,-1}=(i+1)\gamma_{2,1}, \  i\in\mathbb{Z}^+\setminus\{1\}, \quad
\gamma_{-i,1}=-(i+1)\gamma_{2,1}, \ i\in\mathbb{Z}^+\setminus\{1\},$$
$$(1-i)\gamma_{1+i,j}=(1-j)\gamma_{1+j,i}-(i-j)(i+j-1)\gamma_{2,1}.$$

Analyzing these obtained equalities for $j\leq -2,$ we derive
$$\gamma_{i,j}=j(ij-1)\gamma_{2,1}, \ i\neq \pm j, \ i,j\in\mathbb{Z}\setminus\{-1,0,1\}$$

Now  equality (\ref{eq8}) with $j:=i$ and $k:=j$ implies
$$\gamma_{i,i}=(i^3-i)\gamma_{2,1}, \ \ i\in\mathbb{Z}\setminus\{-1,0,1\}.$$

In addition, from (\ref{eq7}) we get $\gamma_{i,-i}=(i^3-i)\gamma_{2,1}, $ $ i\in\mathbb{Z}\setminus\{-1,0,1\}.$

If $\gamma_{2,1}\neq0,$ then by rescaling the basis we can assume $\gamma_{2,1}=1$ and the third algebra of the list of theorem is obtained.

If $\gamma_{2,1}=0$, we have the algebra $(I)$ with $\beta=1.$

\textbf{Case 1.2.} Let $\beta=0$ be. Then (\ref{eq7}) has the form
$\gamma_{j,j}=0, \ j\neq0$ and the have the algebra has the following table of multiplications:
$$\left\{\begin{array}{ll}
[v(k),d_i]=(k+\alpha)v(i+k),&\\[1mm]
[d_0,d_j]=-jd_j,&\ \ j\neq 0,\\[1mm]
[d_j,d_0]=jd_j,&\ \ j\neq 0,\\[1mm]
[d_j,d_j]=0,\\[1mm]
[d_j,d_{-j}]=2jd_0+\gamma_{j,-j}v(-\alpha),&\ \ j\neq 0,\\[1mm]
[d_i,d_j]=(i-j)d_{i+j}+\gamma_{i,j}v(i+j-\alpha),&\ \ i\neq j,\ \ i\neq -j.\\[1mm]
\end{array}\right.$$

Taking the change of $d_0$ as follows $d_{0}'=d_0+\frac{1}{2}\gamma_{1,-1}v(-\alpha)$ we may suppose $\gamma_{1,-1}=0$. Then from (\ref{eq5}) we get $\gamma_{-1,1}=0.$

Substituting in equality (\ref{eq8}) the values of parameters $(i,j,k)$ as $(\pm1,\pm2,\pm1)$
we get the following relations between structure constants:
$$\gamma_{3,-1}=3\gamma_{1,2}, \quad \gamma_{-3,1}=3\gamma_{-1,-2}, \quad \gamma_{-1,2}=\gamma_{-1,3},
\quad \gamma_{1,-2}=2\gamma_{1,-3},$$
$$\gamma_{1,-2}=\gamma_{-1,2}=0, \quad \gamma_{1,3}=-\gamma_{3,1}, \quad \gamma_{-1,-3}=-\gamma_{-3,-1}. $$

Thanks to (\ref{eq6}) the above relations can be written as
$$\gamma_{\pm1,\pm2}=\gamma_{\pm2,\pm1}=\gamma_{\pm1,\pm3}= \gamma_{\pm3,\pm1}=0.$$

Equality (\ref{eq8}) with $i=\mp1, \ j:=i, \ k=\mp1$ with $i\notin\{0,1,-1\}$ together with (\ref{eq6}) lead to
$$(1+i)(i-2)\gamma_{i-1,-1}=(i-1)i\gamma_{i,-1}, \quad (i+2)(i-1)\gamma_{i+1,1}=(i+1)i\gamma_{i,1},$$
from which we obtain
$$\gamma_{-1,i}=\gamma_{i,-1}=0, \quad \gamma_{1,i}=\gamma_{i,1}=0,  \quad i\in \mathbb{Z}\setminus\{-1,0,1\}.$$

Again we consider  equality (\ref{eq8}) with $i=1, \ j:=i, \ k:=j, \ i\neq \pm j, \ i,j\in \mathbb{Z}\setminus\{-1,0,1\}$
and with $j:=1-i, k:=-1, \ i\in \mathbb{Z}^+\setminus\{1,2\},$ we derive
$$(1-i)\gamma_{1+i,j}=(1-j)\gamma_{1+j,i}, \ i\notin\{j,-j\}, \ i,j\in \mathbb{Z}\setminus\{-1,0,1\},$$
$$(2-i)\gamma_{i,-i}=-(i+1)\gamma_{i-1,1-i}, \ i\in \mathbb{Z}^+\setminus\{1,2\}.$$
The simple analysis of these relations bring us to
$$\gamma_{i,j}=0, i,j\in \mathbb{Z}$$ and hence, we obtain the algebra $(I)$ with $\beta=0.$

{\bf Case 2}. Let $\beta= -1$. This case is carried out in a similar way as {\bf Case 1} and we obtain the rest pf the algebras from the list of the theorem. \end{proof}

\begin{rem} From Theorem \ref{thm22} we conclude that in the infinite-dimensional case, unlike to the finite-dimensional case there are many simple Leibniz algebras whose corresponding Lie algebra is a simple Lie algebra and the ideal $\mathcal{I}$ is an irreducible right module over the simple Lie algebra.
\end{rem}

\section{Annihilation of low-dimensional Leibniz cohomology groups of the Witt algebra.}

In this section we prove the coincidence of  the second Lie and Leibniz  cohomologies for Witt algebra. We start with the table of multiplication  of the Witt algebra $\mathcal{W}$. Let us consider a basis $\{d_i, \, i\in \mathbb{Z}\}$ of $\mathcal{W}$ in which the table of multiplications of $\mathcal{W}$ has the form (\ref{equation1}).

The following theorem is true.

\begin{teo} \label{thm33} $HL^i(\mathcal{W},\mathcal{W})=0, \ \ i=1,2.$
\end{teo}
\begin{proof}
Taking into account that for any Lie algebra $\mathcal{G}$ we have $HL^1(\mathcal{G},\mathcal{G})=H^1(\mathcal{G},\mathcal{G})$ (see  \cite{Loday2}) and the result $H^1(\mathcal{W},\mathcal{W})=0$ of paper \cite{Ecker}, we conclude $HL^1(\mathcal{W},\mathcal{W})=0$.

Let $\varphi(d_i,d_j)=\sum\limits_k\alpha_{i,j}^kd_k$ be a Leibniz $2$-cocycle.

From the equalities
$$\Phi(\varphi)(d_0,d_0,d_0)=\Phi(\varphi)(d_i,d_0,d_0)=\Phi(\varphi)(d_0,d_i,d_i)=\Phi(\varphi)(d_i,d_i,d_0)=
\Phi(\varphi)(d_i,d_i,d_j)=0$$
we obtain
$$\varphi(d_i,d_i)=0 , \ i\in \mathbb{Z}, \quad \varphi(d_i,d_j)=-\varphi(d_j,d_i), \,\, i\neq j.$$

Hence, $\varphi$ is a Lie $2$-cocycle, that is, $ZL^2(\mathcal{W},\mathcal{W})=Z^2(\mathcal{W},\mathcal{W}).$

For an arbitrary Leibniz $2$-coboundary $f$ we have have the existence of a linear map $g$.
Since
$$f(d_i,d_j)=[g(d_i),d_j] + [d_i,g(d_j)] - g([d_i,d_j])=-[d_j,f(d_i)]-[g(d_j),d_i]+g([d_j,d_i])=-f(d_j,d_i),$$
we conclude $BL^2(\mathcal{W},\mathcal{W})=B^2(\mathcal{W},\mathcal{W}).$ Therefore, $$HL^2(\mathcal{W},\mathcal{W})=H^2(\mathcal{W},\mathcal{W}).$$
Due to \cite{fialowski} we have $H^2(\mathcal{W},\mathcal{W})$. Hence, $HL^2(\mathcal{W},\mathcal{W})=0$.
\end{proof}


\begin{thebibliography}{99}

\bibitem{Omirov3} {\rm Adashev J.Q., Omirov B.A., Uguz S.}, {Leibniz algebras associated with representations of  Euclidean Lie algebra}, {\it arxiv}:1607.04949. (2016)

\bibitem{Albeverio} {\rm Albeverio S.A., Ayupov Sh.A., Omirov B.A.}, On nilpotent and simple Leibniz algebras, {\it Comm. Algebra.} 33(1) (2005), 159-172.

\bibitem{Ayupov1} {\rm Albeverio S.A., Ayupov Sh.A., Omirov B.A.}, {Cartan subalgebras, weight spaces, and criterion of solvability of finite dimensional Leibniz algebras}, {\it Rev.Mat.Complut.}, 19(1)(2006)183-195.

\bibitem{Ayupov2}{\rm Ayupov Sh.A., Camacho L.M., Khudoyberdiyev A.K., Omirov B.A.}, {Leibniz algebras associated with representations of filiform Lie algebras}, {\it J.Geom.Phys.} 98(2015)181-195.

\bibitem{simple}{\rm Ayupov Sh.A., Kudaybergenov K.K., Omirov B.A., Zhao K.}, {Semisimple Leibniz algebras and their derivations and automorphisms}, {\it arxiv.}, 1708.08082v1.

\bibitem{Ayupov3} {\rm Ayupov Sh.A., Omirov B.A.}, {On Leibniz algebras}, {\it Algebra and operator theory (Tashkent, 1997), Kluwer Acad. Publ., Dordrecht}, 1998, pp. 1-12.

\bibitem{BarnesEngel} {\rm Barnes, D.} On Engel's Theorem for Leibniz Algebras, {\it Comm. Algebra.} 40 (4) (2012), 1388-1389.

\bibitem{Barnes}{\rm Barnes D.}, {On Levi's Theorem for Leibniz algebra}, {\it Bull. of the Australian Math. Soc.}, 86 (2012), 184-185.

\bibitem{Bosko} {\rm Bosko, L., A. Hedges, J.T. Hird, N. Schwartz, K. Stagg.} Jacobson's refinement of Engel's theorem for Leibniz algebras, {\it Involve.} 4 (3) (2011), 293-296.

\bibitem{Omirov1}{\rm Calder\'{o}n A.J., Camacho L.M., Omirov B.A.,} {Leibniz algebras of Heisenberg type}, {\it J.Algebra.}, 452 (2016) 427-447.

\bibitem{diamond} {\rm Camacho L.M., Karimjanov I.A., Ladra M., Omirov B.A.}, {Leibniz algebras constructed by representations of General Diamond Lie algebras}, Bull. Malays. Math. Sci. Soc. (2017).  https://doi.org/10.1007/s40840-017-0541-5

\bibitem{cartan}{Cartan E.,} {Les groupes de transformations continus, infinis, simples.} {\it Ann. Sci. Ecole Norm. Sup.}, 26 (1909) 93-161.

\bibitem{Ecker} {\rm Ecker J., Schlichenmaier M.}, {The vanishing of low-dimensional cohomology groups of the Witt and the Virasoro algebra}, {\it arxiv.}, 1707.06106v1.

\bibitem{WittRepr} Eswara Rao S., Representations of Witt algebras, \textit{Publ. Res. Inst. Math. Sci.} 30 (1994) 191 201.

\bibitem{fialowski} {\rm Fialowski A., Schlichenmaier M.} {Global deformations of the Witt algebra of Krichever-Noviker type,} {\it Communications in Contemparary Mathematics.,} 5(6) (2003) 921-945.

\bibitem{Hump} {\rm Humphreys J.E.}, {Introduction to Lie algebras and Representation theory}, {\it Springer-Verlag New York}, 1972.

\bibitem{Jacob} {\rm Jacobson N.,} {Lie algebras,} 340p. {\it Interscience Publishers, Wiley, New York}, (1962).

\bibitem{KacV}{\rm Kac V., Raina A.,} {Bombay lectures on highest weight representations of infinite-dimensional Lie algebras,} {\it World Sci.Singapore}, 1987.

\bibitem{KW}  {\rm Kinyon M.K., Weinstein A.,} {Leibniz algebras, Courant algebroids, and multiplications on reductive homogeneous spaces,} {\it Amer. J. Math.}, 123~(3) (2001) 525-550.

\bibitem{Loday} {\rm Loday J.-L.} {Une version non commutative des alg$\grave{e}$bres de Lie: les alg$\grave{e}$bres de Leibniz,} {\it Enseign.Math.}, (2) 39 (3-4) (1993) 269-293.

\bibitem{Loday2} {\rm Loday J.-L., Pirashvili T.}, {Universal envoloping algebras of Leibniz algebras and (co)homology},
{\it Math.Ann} 296, (1993), 139-158-748.

\bibitem{Zhao2} {\rm L$\ddot{u}$ R., Guo X., Zhao K.}, {Irreducible modules over the Virasoro algebra}, {\it Doc.Math.}, 16 (2011), 709-721.


\bibitem{OmirovCartan} {\rm Omirov, B.,} Conjugacy of Cartan subalgebras of complex finite dimensional Leibniz algebras, {\it J. Algebra.} 302(2) (2006), 887-896.

\bibitem{Omirov2} {\rm Uguz S., Karimjanov I.A., Omirov B.A.}, {Leibniz algebras associated with representations of the Diamond Lie algebra}, Algebr Represent Theor, 20 (2017) 175-195. https://doi.org/10.1007/s10468-016-9636-1

\bibitem{Zhao1} {\rm Zhao K.}, {Weight modules over generalized Witt algebras with 1-dimensional weight spaces}, {\it Forum Math.} Vol.16 (2004), 725-748.
\end{thebibliography}
\end{document}